\documentclass{ksiam}
\usepackage{blindtext}
\usepackage{hyperref}
\usepackage{amsmath, amssymb, latexsym}
\usepackage{epsfig}
\DeclareMathOperator{\supp}{supp}

\newtheorem{definition}{Definition}[section]

\newtheorem{remark}{Remark}[section]

\newtheorem{theorem}{Theorem}[section]

\newtheorem{lemma}{Lemma}[section]

\newtheorem{corollary }{Corollary}
\newtheorem{Lad-Ur}{Ladyzhenskaya-Uraltceva iterative Lemma}
\newtheorem{hypotheses}{Hypothesis}
\newtheorem{exmpl}{Example}

\AtBeginDocument{
\addtolength\abovedisplayskip{-0.25\baselineskip}
 \addtolength\belowdisplayskip{-0.25\baselineskip}
 \addtolength\abovedisplayshortskip{-0.25\baselineskip}
  \addtolength\belowdisplayshortskip{-0.25\baselineskip}
}

\newcommand{\ind}{\hbox{\rm 1\hskip -4.5pt 1}}

\newtheorem{proposition}{Proposition}[theorem]

\usepackage{setspace}
\usepackage[left=1 in,top=1 in,right=1 in,bottom=1 in]{geometry}
\begin{document}

\title[Finite speed of propagation in Generalized Einstein paradigm]
{Einstein's model of "the movement of small particles in a stationary liquid" revisited: finite propagation speed}

\author[I. G. Hevage]{Isanka Garli Hevage $^{1}$}
\address{$^{1}$Department of Mathematics and Statistics, Texas Tech University, TX 79409--1042, U. S. A. 
\newline\\
$^{2}$Department of Mathematics, University of Swansea,
Fabian Way, Swansea SA1 8EN, UK} 
\email{isankaupul.garlihevage@ttu.edu, akif.ibraguimov@ttu.edu, zsobol@gmail.com}


\author[A. Ibraguimov]{ Akif Ibragimov $^{1}$}
\author[Z. Sobol]{Zeev Sobol $^{2}$}


\begin{abstract}
The aforementioned celebrated model, though a breakthrough in Stochastic processes and a great step toward construction of the Brownian motion, leads to a paradox: infinite propagation speed and violation of the 2nd law of thermodynamics. We adapt the model by assuming the diffusion matrix dependent of the concentration of particles, rather than constant it was up to Einstein, and prove a finite propagation speed under assumption of a qualified decrease of the diffusion for small concentration. The method involves a nonlinear degenerated parabolic PDE in divergent form, a parabolic Sobolev-type inequality and the Ladyzhenskaya-Ural'tseva iteration lemma.
\end{abstract}
\maketitle

\section{Introduction}

In his celebrated paper \cite{Einstein56} Einstein models a particles movement as a random walk  of a step time $\tau$ and (random) displacement (\textit{free jump}) $\Delta$, of a symmetric distribution independent of a point and time of observation. The walk is restricted by the mass conservation law written for the concentration function. Using the argument eventually developed into the Ito formula, Einstein shows that the concentration function $u$, being a density of the distribution of particles, satisfies the classical Heat equation (the forward one, since the random walk has a reversible law). (See \cite{Einstein56} \S 4 or Section \ref{Gen-Ein-Para} below for details). 

Though being a revolutionary paper in stochastic processes, and a decisive step toward construction of Brownian motion, this model, however, leads to a Physical paradox. Being a solution to the Heat equation, the concentration $u$ allows for a void volume instantly reaching a positive concentration of particles. Moreover, as the \textit{free jump}s process is reversible, the model allows for all particles, with a wonderful coherence,  instantly  concentrating in a small volume. This contradicts the second law of thermodynamics, as well as demonstrate an infinite propagation speed. 

The aim of this note is as follows: to adapt Einstein's model of \textit{free jump}s, with a random walk replaced with a diffusion process, so as to get free of this paradox. We suppose that the villain of the piece  is the assumption that diffusion $a=\sigma^2/\tau$ is constant ($\sigma^2$ being the variance of $\Delta$). Keeping with the Einstein assumption of an isotropic stationary media, we assume no drifts or sources, but the diffusion matrix $a=a(u)$ dependent of the concentration only, rather than the constant it was up to Einstein. Our aim is to choose $a(u)$ so that concentration function $u$ demonstrates a \textit{finite propagation speed}, that is, if a neighbourhood of a point is void of the particles  at time $T$, then a (smaller) neighbourhood of the same point have been void of the particles for some time preceding $T$.

Finite propagation speed was demonstrated first by G.I.Barenblatt for a degenerate porous media equation (see \cite{Barenblatt-96, Isakif}). (The origin of the porous medium equation differs from the Einstein model.) The example of porous media equation hints that a finite propagation speed appears if small concentration implies small diffusion. Hence we assume a diffusion matrix $\{a_{ij}\}$ to be $a(u)Id$, with a positive continuous function $a$ of concentration $u$, such that $a(0)=0$ (see Hypothesis \ref{tau}). This reflects the case of a higher medium resistance for a small numbers of particles. One may think of hot particles heating the liquid when in numbers, and increasing its permeability. 

Note that concentration function $u$ can be considered in two ways: first, as a function, and second, as a density of the particles distribution. The former approach leads to the backward Kolmogorov equation while the latter leads to the forward one. In this note we follow the first approach, considering a backward equation \eqref{post-Taylor} for $u$. However, reversing the time brings the equation to a regular forward form. 

Nevertheless, this is not the main trick we play with time. The main is a passage to an inner local time of the process, dependent of the concentration as well. On practice, it is multiplication of the equation by some weight $h(u)>0$ (meaning a local change of time 
$t\to \frac t{h(u)}$). This regularises the coefficients of the equation and allows us to consider its divergent form and apply a rich technique of weak solutions.

The main result on finite speed of propagation is  Theorem \ref{Main-T}. It states that the concentration $u$ demonstrates a finite propagation speed, provided the following.
\begin{align}
    &\limsup\limits_{u\to 0}a(u)I(u)<\infty\label{a} , 
    \\
    &\mbox{there exist } c,\mu>0 \mbox{ such that } \
    a(u)I^\mu(u) \geq c a(v)I^\mu(v),\ 0<u<v,
    \label{adecr}
\end{align}
\begin{equation}\label{I-s}
\text{with}\qquad I(u) = \int_u^\infty \frac{dv}{v\,a(v)}.
\end{equation}

The article is organized as follows. In Section \ref{Gen-Ein-Para}, we consider generalization of the classical Einstein model of an $N$-dim Brownian motion,
with source term and general diffusion coefficients. 
We use the generic mass conservation law \eqref{cons} to derive a partial differential inequality \eqref{M-2} for concentration $u$ in Section \ref{Nonlin}. In the same section we
pass to the inner local time in \eqref{L-OP}, and to weak solutions formulation in \eqref{weak}. In Section \ref{Auxiliary} we bring various auxiliary results necessary for the proof of Theorem \ref{Main-T}. They include an ingenious non-linear parabolic Sobolev-type inequality \eqref{lemma-2-R} and the celebrated iteration lemma by Ladyzhenskaya and Ural'tseva \ref{Lady-lemma}. In Section \ref{localization} we prove the finite propagation speed property. Section \ref{Exmp-P-F} focuses on various examples of generalised Einstein models satisfying conditions of the main theorem.

\section {Generalized Einstein model}\label{Gen-Ein-Para}

For a space-time point of observation $Z=(x,s)\in \mathbb{R}^N\times[0,\infty)$, consider an $\mathbb{R}^N$-valued random \textit{free jumps process} $\vec{\Delta}^Z$,
\[
\vec{\Delta}^Z \triangleq (\vec{\Delta}_t)_{t\geq s}^Z 
,
\]
describing a interaction-free displacement of a particle off $Z$ (so $\vec{\Delta}^Z_s=0$ a.s.). We assume the following extension of the axioms of classical Brownian motion.
\begin{hypotheses}\label{diffusion}   \normalfont\textit{Free jumps process}  $(Z,x+\vec{\Delta}^Z)_{Z\in \mathbb{R}^N\times \mathbb{R}_+}$ is a diffusion process with a state space $\mathbb{R}^N$ (see, e.g, \cite{Skorohod} Ch.2 Sect. 1.3). In particular it means that \textit{free jumps process} is a Markov process with continuous trajectories and that the following axioms hold:
\begin{enumerate}
    \item \normalfont {Whole universe axiom:}
 \begin{align}\label{uni-ax}
      \mathbb{P}\{\vec{\Delta}^Z_t\in \mathbb{R}^N \}= 1,\quad t>0.
 \end{align}
 \item \normalfont {Trajectory continuity axiom:}
 for all $\epsilon>0$,  uniformly in $Z\in \mathbb{R}^N\times \mathbb{R}_+$,
\begin{equation}\label{cont}
  \frac1T\sup\limits_{0<t<T}\,\mathbb{P}\{\| \vec{\Delta}^Z_t\|>\epsilon\}
\to 0 \quad\text{as}\quad T\to0.
\end{equation}
See \cite{Skorohod} Ch.2 Sect. 1.1.3 for the proof of  continuity of a.a. trajectories of the process.
\item  \normalfont {Diffusion coefficients axiom:} for some (hence all) $\epsilon>0$, uniformly in $Z\in \mathbb{R}^N\times \mathbb{R}_+$,
\begin{align}
    \frac1T\sup\limits_{0<t<T}\,\mathbb{E}\left(\vec{\Delta}^{Z}_t \times\vec{\Delta}^{Z}_t\right) \ind_{\{|\vec{\Delta}^{Z}_t|<\epsilon\}} & \to a(Z),\label{diff}\\
    \frac1T\sup\limits_{0<t<T}\,\mathbb{E}\vec{\Delta}^{X}_t \ind_{\{|\vec{\Delta}^{Z}_t|<\epsilon\}} & \to b(Z),\label{trans}
\end{align}
as $t\to0$. Matrix $a(Z)=\left\{a_{ij}(Z)\right\}_{ij=1}^N$ is referred as \textit{diffusion matrix}, and vector $b(Z)=\left\{b_{k}(Z)\right\}_{k=1}^N$ is referred as \textit{transport(drift) coefficient} (cf \cite{Skorohod}, Part II, Definition 1.3.1).
\item We assume the lifetime $\tau^Z$ of $\vec{\Delta}^Z$ locally bounded from below: for every bounded domain $G\Subset \mathbb{R}^N\times \mathbb{R}_+$
    \begin{equation}\label{bounded-conc}
     \tau_G \triangleq \inf\limits_{Z\in G} \tau^Z >0.
    \end{equation}
\end{enumerate}
\end{hypotheses}
An important property of a diffusion process is existence of a correspondent diffusion differential operator, as it is shown in the following lemma (see \cite{Skorohod}, Part II, Lemma 1.3.1.)
\begin{lemma}[Diffusion operator]\label{DO}
Let $\phi$ be a bounded twice differentiable function on $\mathbb{R}^N$. (By the latter we understand that $\phi\in C^1(\mathbb{R}^N)$ and $\nabla \phi$ is absolutely continuous on any ray.) Then, locally uniformly in $Z\in \mathbb{R}^N\times \mathbb{R}_+$,
\begin{equation}\label{rhs1}
\begin{aligned}
\frac1T\sup\limits_{0<t<T}\,\left(\mathbb{E}\,\phi(x + \vec{\Delta}^Z_t) - \phi(x)\right) &
\to  \frac{1}{2}\sum_{i,j=1}^N  a_{ij}(Z){\phi}_{x_{i} x_{j}} (x) -
\sum_{k=1}^N  b_{k}(Z){\phi}_{x_{k}} (x),\\
\text{as}\quad & T\to0.
 \end{aligned}
\end{equation}
\end{lemma}
The second component of the Einstein paradigm is the mass (concentration) conservation law for a concentration function $u=u(x,s)$. The law reflects
the idea of a spreading inkblot: the concentration $u(x,s-t)$ of particles at a space-time point $(x,s-t)$ equals to the total concentration of particles at points $(x + \vec{\Delta}^{(x,s-t)}_t ,s)$ reached by particles by the time $s$ plus/minus the number of particles consumed/produced in the way by the time period $[s - t, s]$, with a consumption-production reported by its flux $M$, $M(Z)$ being the quantity consumed/produced at a space-time point $Z$ (a reader may think of a spreading and drying ink).
\begin{hypotheses} \ {\normalfont (Axiom of Mass Conservation)}
\begin{equation}\label{cons}
u(x, s - t) =
\mathbb{E}\, u(x + \vec{\Delta}^{(x,s-t)}_t, s)
+ \int\limits_0^t \mathbb{E} M(x + \vec{\Delta}^{(x,s-t)}_r, s - t +r)dr
\end{equation}
\end{hypotheses}
If we assume some regularity of $u$, we arrive at the following differential equation (see \cite{Skorohod} Ch.2 Theorem 1.3.1 for the idea of the proof).

\begin{theorem}\label{DEbackward}
  Assume that $u(s)$ is twice differentiable in the spatial variable (in the sense of \ref{DO}) for every $s\in \mathbb{R}_+$. Then $u$ is absolutely continuous in time $s$ and satisfies equation
  \begin{equation}
- u_s(Z) =
 \frac{1}{2}\sum_{i,j=1}^N  a_{ij}(Z){u}_{x_{i} x_{j}}(Z)
 + \sum_{k=1}^N  b_{k}(Z){u}_{x_{k}}(Z)
  + M(Z).
 \label{post-Taylor}
\end{equation}
\end{theorem}

\begin{remark}\label{einstein}
  Einstein in \cite{Einstein56} made a strong assumption that the distribution of \emph{free jump} $\vec{\Delta}^Z$ does not depend on $Z$. This yielded diffusion coefficients $a$ and $b$ being constant and allowed Einstein to write down \eqref{cons} in the forward way (cf. \cite{Einstein56}, p.14). So by choice $a=Id$, $b=0$ and $M=0$ he reduced the problem to the standard heat equation. However, this leads to violation of the second law of thermodynamics. Indeed, with $a=Id$, $b=0$ and $M=0$, equation \eqref{post-Taylor}
  becomes $-u_s = \tfrac12 \Delta u$. Then a solution $u$,
  \[
  u(x,s) = \left(2\pi(t-s)\right)^{-\frac N2}\exp\left\{-\frac{|x|^2}{2(t-s)}\right\}, \qquad x\in\mathbb{R}^N,\ 0\leq s<t
  \]
  is positive on $\mathbb{R}^N\times [0,t)$ while $u(x,s)\to0$ as $s\uparrow t$ for all $x\ne0$. Thus, all particles instantly concentrate at zero at time $t$, demonstrating not only an impossible coherence but also an infinite speed. This paradox is the main motivation of our study.
\end{remark}

\section{Nonlinear Degenerate Inequality} \label{Nonlin}
In this section we make final assumptions on the structure of the Generalised Einstein model and state the main problem of the note. Our aim is to establish some conditions on diffusion coefficients $a$ and $b$ and consumption-production flux $M$ in \eqref{post-Taylor} allowing us to escape the paradox described in Remark \ref{einstein}. First we bring \eqref{post-Taylor} into a forward form, more used for specialists in parabolic PDEs. We will fix a domain $\Omega$ and a time-horizon $T>0$, and consider  \eqref{post-Taylor} on
$\Omega_T \triangleq \Omega \times (0,T)$. With a change of variables $t=T-s$, equation \eqref{post-Taylor} takes the following form:
\begin{align}
   u_t(Z) =
 \frac{1}{2}\sum_{i,j=1}^N  a_{ij}(Z){u}_{x_{i} x_{j}}(Z)
 + \sum_{k=1}^N  b_{k}(Z){u}_{x_{k}}(Z)
  + M(Z), \qquad Z\in \Omega_T
\label{M-1}
\end{align}
We seek to establish the following property of $u$.
\begin{definition}[Finite propagation speed]\label{FPS}
A function $u\geq0$ on $\Omega_T$ is said to enjoy a finite propagation speed in $t$ if, for any open ball
$B\subset \Omega$ and any $\epsilon\in(0,1)$, there exists $s\in (0,T]$ (which might depend on $B$, $\epsilon$ and $u$), such that, given $u(x,0)=0$
for all $x\in B$, one has $u(x,t)=0$ for all $(x,t) \in \epsilon B\times[0,s]$.
\end{definition}
Obviously, if $u$ enjoys a finite propagation speed in $t$, then particles demonstrate finite speed in $s$, too. Moreover, a point $x$ is void of particles at time $T$ only if it was such for some time before. Thus, a finite propagation speed resolves the paradox of the Einstein model.

In this note we consider a simple model to study the very essence of the phenomenon. Hence the following assumptions  in the form of Hypothesis
\begin{hypotheses}\label{tau}\text{}
\begin{description}
\item[No drift:] $b=0$;
\item[No consumption:] $M\leq 0$ (recall that $M$ has come from \eqref{cons} which is a backward equation);
\item[Basic diffusion matrix:] $a_{ij}=2a(u)\delta_{ij}$, $i,j=1,2,\ldots,N$, with some $a\in C[0,\infty)$, $a(u)>0$ for $u>0$ and $a(0)=0$. Hypothesis \ref{diffusion}(iv) suggests that $u$ is locally bounded.  Hence we are free in choosing a behaviour of $a$ at infinity. So we
assume $a$ such that, with $I(u)$ defined as in \eqref{I-s},
\begin{align}
 &I(u) <\infty, \qquad u>0; \\
 &\limsup\limits_{u\to\infty}a(u)I(u)<\infty.\label{infty}
\end{align}
Note that $I(u)\to\infty$ as $u\to0$ and that \eqref{a} implies that $aI$ is a bounded continuous function.
\end{description}
\end{hypotheses}
The concept of a lower diffusion speed for a lower concentration of particles (the only hope to obtain a finite propagation speed) reflects the case of a higher medium resistance for a smaller numbers of particles.

With Hypothesis \ref{tau} equation \eqref{M-1} converts to the following inequality
\begin{align}
    u_t \leq a(u) \Delta u, \qquad \text{in}\quad \Omega_T. \label{M-2}
\end{align}
We will look for solution in the following class.
\begin{definition}\label{solution}
A weak positive bounded solution $u$ to \eqref{M-2} is meant to be a positive 
$u \in L^{\infty}(\Omega_T)$ such that $\nabla u\in L^{2}_{loc}\left( \Omega_T\right)$, and map $[0,T]\ni t\mapsto u(t)\in L_{loc}^0(\Omega)$ is continuous. (Hence, for any $Q\in C(\mathbb{R}_+)$ and $p\geq1$, map $[0,T]\ni t\mapsto Q\big(u(t)\big)\in L_{loc}^p(\Omega)$ is continuous as well.)
\end{definition}
We do not require $u$ neither differentiable in $t$ not twice differentiable in $x$ as it is not a classic solution. Our approach is different: we perform a local time change $t\to \frac t{h(u)}$, that is,
we multiply \eqref{M-2} by a strictly positive function $h(u)$, with $F=ha$ and primitive $H$ of $h$. Hence we obtain an equivalent form of \eqref{M-2}:
\begin{align}
 [ H(u)]_t -   F(u)\Delta u \leq 0 \quad \mbox{in } \Omega_T. \label{L-OP}
\end{align}
The first properties we require  of $h,H$ and $F$ are as in the following definition.
\begin{definition} \label{FH-def}
\text{}
\begin{itemize}
\item
     Let $h\in C(0,\infty)$, $h > 0$, integrable at 0, and such that $ha$ is a monotone locally Lipschitz function, $h(u)a(u)\to 0$ as $u\to 0$. Let
     \begin{align}
          H(u) \triangleq\int_0^u h(s) \ ds.   \label{H-F-def}
     \end{align}
    \item   Let 
    \begin{align}
         F(u) \triangleq h(u)a(u)   \label{F-def}
     \end{align}
    By the preceding, $F(0)=0$, and $F$ is differentiable on $(0,\infty)$ with a locally bounded derivative $F^{'}\geq0$. \label{F-def}
\end{itemize} 
\end{definition}
Note that $H(u),\nabla F(u)\in L^2_{loc}(\Omega_T)$ if $u$ is as in Definition \ref{solution}. Hence we understand $F(u)\Delta u$ and $H(u)]_t$ in the weak sense, by which we mean the following. 
For every
$\phi\in Lip_c(\Omega_T)$,
\begin{equation*}
\begin{aligned}
-\int_{\Omega_T} \phi F(u)\Delta u dx\,dt & =
\int_{\Omega_T} \nabla u \nabla \left (F(u)\phi\right)dx\,dt\\
 \int_{\Omega_T} \phi [H(u)]_t dx\,dt
 & = - \int_{\Omega_T} \phi_t H(u) dx\,dt.
\end{aligned}
\end{equation*}
Thus, by a (positive bounded) solution $u$ to \eqref{M-2} we mean $u$ as in Definition \ref{solution} satisfying \eqref{L-OP} in the following way:
\begin{equation}\label{weak}
 \int_{\Omega_T} \nabla u \nabla \left (F(u)\phi\right)dx\,dt
 \leq 
 \int_{\Omega_T} \phi_t H(u) dx\,dt, \qquad
\phi\in Lip_c(\Omega_T),\ \phi\geq0. 
\end{equation}

\section{Auxiliary results}\label{Auxiliary}
In this section we introduce some structure functions and bring auxiliary results necessary for proving the finite propagation speed for $u$ satisfying \eqref{weak}. 

In addition to structure functions $F$ and $H$,
let
\begin{align}
          G(s) \triangleq \int_{0}^{s} \sqrt{F^{'}(\sigma)} \ d\sigma. \label{G-def}.
     \end{align}
    So  $\sqrt{F^{'}(u)} = G^{'}(u)$.
\begin{remark}\label{G-F-2}
Note that
$ \displaystyle
0\leq G(u)=\int_0^u G^{'}(s)ds \leq \sqrt{u\int_0^u F^{'}(s)ds} = \sqrt{uF(u)}$, and all functions $F,G$ and $H$ are increasing on closed interval.
\end{remark}
Now we will make a set up of functions $H$ and hence $h$, $F$ and $G$.

\begin{definition}\label{PIH}
For some $\Lambda>0$, and $I$ as in \eqref{I-s}, define
\begin{align}\label{H-choice}
 H(s) \triangleq
    \left[\Lambda I(s)\right]^{-\frac{1}{\Lambda}} =
    \left(\Lambda \int_{s}^{\infty} \frac{1}{\tau a(\tau)} d\tau\right)^{-\frac{1}{\Lambda}},
    ~ s>0.
\end{align}
\end{definition}
\begin{remark} \label{new-def}
In \eqref{H-choice}, $H(s) \to 0$ as $s\to 0$. Moreover,
\begin{align}
    h(s) & = \frac{1}{sa(s)}\left[  \Lambda I(s) \right]^{-\frac{1}{\Lambda}-1}
    = \frac{1}{sa(s)} H^{(\Lambda+1)}(s)
    \ , ~ s>0. \label{h-ret}\\
    F(s)& =h(s)a(s)=\frac{1}{s} H^{\Lambda+1}(s)=\left(\Lambda s^{\frac{\Lambda}{1+\Lambda}}I(s) \right)^{-\frac{1}{\Lambda}-1}.
    \label{F-result}
\end{align}
Finally, with $\lambda\in(0,2)$ such that $\Lambda + 1= \dfrac{2}{\lambda}$ one has
\begin{equation}\label{A-2}
 \left(sF(s)\right)^{\frac\lambda2}\leq H(s)
\end{equation}
\end{remark}
Note that in general $F$ defined by \eqref{F-result} is neither monotone increasing nor vanishes at zero. So it does not automatically satisfy Definition \ref{FH-def}. This properties and parabolic Sobolev inequality is the subject of the following proposition.
\begin{proposition}\label{parabolic_Sobolev}
Let \eqref{a} hold. Choose $\Lambda$,
\begin{equation}\label{Lambda}
    0< \Lambda < \frac1{(\sup\limits_u a(u)I(u))_+},
\end{equation}
and $H,F,G$ as in \eqref{H-choice}, \eqref{F-result} and
\eqref{G-def}, respectively.

Then $F'>0$ and $F(0)=0$. Moreover, the following parabolic Sobolev-type inequality holds: for all
\begin{itemize}
    \item domain $\Omega\subset \mathbb{R}^N$ and $T>0$;
    \item $\theta\in Lip_c(\Omega)$, $0\leq\theta\leq1$ and
$K\subset \{\theta=1\}$;
\item $u\in L^\infty_{loc}(\Omega\times[0,T])$,
$\nabla u\in L^2_{loc}(\Omega \times (0,T)) $ and $\nabla G(u) \in L^2_{loc}(\Omega \times [0,T])$;
\end{itemize}
one has
\begin{align}
\int_{0}^{t} \int_{K} G^{2}(u) dxdt
\leq S^{k(1+j)}
t^{1-(1+j)k}
\left[  \sup_{0\leq \tau \leq t} \int_{\Omega} \theta^{2} H (u(\tau)) dx
+
 \int_{0}^{t}\int_{\Omega} |\nabla(\theta u)|^2 dx d\tau
    \right]^{1+jk}.\label{lemma-2-R}
\end{align}
Here $j=\frac2{N-2}$,
$k=\frac{\Lambda}{\Lambda + j + j\Lambda}$,
and $S$ is a constant in  the Sobolev inequality
\begin{equation}
\|\psi\|_{L^{2+2j}}^{2}
\leq
S \| \nabla \psi \|_{L^{2}}^{2}.\label{Gil-Nir}
\end{equation}
\end{proposition}

Proposition \ref{parabolic_Sobolev} is one of the three auxiliary results we need for demonstrating a `finite propagation speed for solutions of \eqref{M-2}. The second one deals with the left-hand side of \eqref{weak}.

\begin{proposition}\label{cacciopoly}
Let $\eqref{a}$-\eqref{adecr}. Choose 
$\Lambda$ as in \eqref{Lambda} and $H,F,G$ as in \eqref{H-choice}, \eqref{F-result} and
\eqref{G-def}, respectively.

Then there exists $C\geq1$ such that
\begin{equation}
\nabla u\cdot \nabla\left( \theta^{2} F(u)\right) \geq
\frac{1}{2}| \nabla (\theta G(u))|^{2} - C G^{2}(u)\vert\nabla \theta \vert^{2}.
    \label{lemma-1-R}
\end{equation}
for any measurable $u$, $\nabla u$ and Lipschitz continuous $\theta$.
\end{proposition}
The last one is the celebrated Ladyzhenskaya-Uraltceva iteration lemma (see
 \cite{LU} Ch 2. Lemma 4.7).
\begin{lemma}\label{Lady-lemma}
Let sequence $y_n$ for $n=0,1,2,...$, be nonnegative sequence satisfying the recursion inequality,
$ \displaystyle  y_{n+1}\leq c\text{ }b^n \text{ }y_n^{1+\delta} $ with some constants $ c ,\delta > 0 \text{ and } b\geq 1$. Then
\[ \displaystyle y_n \leq c^{\frac{(1+\delta)^n -1}{\delta}}\text{ } b^{\frac{(1+\delta)^n -1}{\delta^2} -\frac{n}{\delta}}\text{ }y_0^{(1+\delta)^n}.\]
In particular $ \displaystyle \text{if} \quad y_0 \leq \theta_L = c^\frac{-1}{\delta} \text{ } b^\frac{-1}{\delta^2} \text{ and } {b > 1}$, then  $y_n \leq \theta\text{ } b^{\frac{-n}{\delta}}$
and consequently, \[ \displaystyle y_n \rightarrow 0 \text{ when } n\rightarrow \infty.\]
\end{lemma}

By the end of the section we will be proving and commenting Propositions \ref{parabolic_Sobolev} and \ref{cacciopoly}. The reader not interested in the technique can pass to the next section.

The proofs will be split into several lemmas, to facilitate the reading.

\begin{lemma}\label{P-2}
Let \eqref{a} hold. Choose $\Lambda$ as in \eqref{Lambda} and $F$ as in \eqref{F-result}. Then $F' > 0$, $F(0)=0$.
\end{lemma}
\begin{proof}
Note that \eqref{Lambda} implies
\[
\frac{\Lambda + 1}\Lambda > \sup\limits_u a(u)I(u).
\]

Let $F$ be as in \eqref{F-result}. We show that $F$ is increasing. It is equivalent to demonstrating that function
\begin{equation*}s \mapsto \displaystyle{ s^{\frac{\Lambda}{1+\Lambda}}I(s)},
\end{equation*}
decreases. And so it is since
\begin{align*}
\frac{d}{ds}\left( s^{\frac{\Lambda}{1+\Lambda}}I(s)\right)
=
\left( \frac{\Lambda}{\Lambda+1} \frac{1}{a(s)}\right)\left( a(s)I(s) -\frac{\Lambda+1}{\Lambda}\right)s^{-\frac{1}{\Lambda+1}}<0.
\end{align*}

Next, we show that
$s^{\frac{\Lambda}{1+\Lambda}}I(s) \rightarrow \infty$ as $s\rightarrow 0$, which implies $\displaystyle \lim_{s \rightarrow 0}F(s) = 0$.
By \eqref{a}, $\limsup\limits_{s\to0}a(s)I(s)=\alpha<\infty$. Hence, for every $\epsilon > 0 $ there exists $s_{\epsilon} \in (0,M]$ such that $ a(s)I(s) < \alpha+\epsilon$ for $s\in(0,s_\epsilon)$. This yields the following inequalities.
\begin{align*}
 I(s) < (\alpha + \epsilon)\frac{1}{a(s)}  &  = -(\alpha+\epsilon)s I^{'}(s),\\
    \frac{d}{ds} \ln I(s)  & < -\frac{1}{\alpha+\epsilon} \cdot \frac{1}{s},\\
I(s) & >  I(s_{\epsilon})\left(\frac{s_{\epsilon}}{s}\right)^{\frac{1}{\alpha+\epsilon}} \text{ for }  0 < s < s_\epsilon, \\
s^{\frac{\Lambda}{\Lambda+1}}I(s) & \geq
I(s_{\epsilon})s_{\epsilon}^{\frac{1}{\alpha+\epsilon}}\times s^{\frac{\Lambda}{\Lambda+1}-\frac{1}{\alpha+\epsilon}}.
\end{align*}
Since $\frac{\Lambda +1}\Lambda > \max PI \geq \alpha$, one can choose $\epsilon>0$ such that $\frac{\Lambda +1}{\Lambda} > \alpha + \epsilon \implies \frac{\Lambda}{\Lambda+1}-\frac{1}{\alpha+\epsilon}<0$. Hence $
 s^{\frac{\Lambda}{\Lambda+1}}I(s) \to \infty$ as $s\to 0.$
\end{proof}

\begin{lemma}\label{lemma-2}
Inequality \ref{A-2} implies \eqref{lemma-2-R}.
\end{lemma}
\begin{proof}
Let $\lambda = \frac2{\Lambda + 1}$.
By Remark \ref{G-F-2} and \eqref{A-2},  $G(s)^{\lambda}\leq H(s)$. Note that
\[
\lambda = \frac{2 - 2(1+j)k}{1-k}.
\]
So
\begin{align*}
G^{2} (u) \leq  G^{2(1+j)k}(u)H^{1-k}(u). 
\end{align*}
Integrate both side of the latter over $ K \times (0,t)$ to obtain the following.
\begin{align*}
\int_{0}^{t} \int_{K} G^{2}(u) dxdt
\leq &
\  \int_{0}^{t}\int_{K}G^{2(1+j)k}(u)H^{1-k}(u) dxd\tau \nonumber \\
\leq &
\ \int_{0}^{t}\int_{K} \left(|\theta G(u)|^{2(1+j)}\right)^{k}
\left( \theta^{2}H(u)\right)^{1-k} dx d\tau\\
\leq &
\ \int_{0}^{t}
    \left[\int_{\Omega} |\theta G(u)|^{2(1+j)} dx\right ]^{k}
\left[ \int_{\Omega}\theta^{2}H(u) dx\right]^{1-k}   d\tau\\
\leq&
\ S^{k(1+j)}
 \int_{0}^{t}
    \left[\int_{\Omega} |\nabla(\theta G(u))|^{2} dx\right]^{(1+j)k}d\tau \left[\sup_{0\leq \tau \leq t}\int_{\Omega}\theta^{2}H(u) dx\right]^{1-k}, \nonumber 
\end{align*}
by \eqref{Gil-Nir}.  Note the inequality $ x^{v}y^{w}\leq (x+y)^{v+w} \ ; \ x,y,v,w > 0 $. Indeed
\[
x^{\frac v{v+w}}y^{\frac w{v+w}} \leq
\frac v{v+w} x + \frac w{v+w} y 
\leq \frac{\max\{v,w\}}{v+w} (x+y)<x+y
\]
We apply these together with the Holder inequality for time integral, to get the following
\begin{align}
\int_{0}^{t} \int_{K}  G^{2}(u) & dxdt \nonumber   \leq    S^{k(1+j)} \left[\sup_{0\leq \tau \leq t}\int_{\Omega}\theta^{2}H(u) dx\right]^{1-k} t^{1-k(1+j)}
 \left[\int_{0}^{t} \int_{\Omega} |\nabla(\theta G(u))|^{2} dx  d\tau\right]^{(1+j)k} \nonumber \\
& \leq
 S^{k(1+j)} t^{1-k(1+j)} \left[\sup_{0\leq \tau \leq t}\int_{\Omega}\theta^{2}H(u) dx + \int_{0}^{t} \int_{\Omega} |\nabla(\theta G(u))|^{2} dx  d\tau\right]^{1+jk}. \nonumber
\end{align}
\end{proof}

The core of the proof of Proposition \ref{cacciopoly} is the following lemma.
\begin{lemma} \label{lemma-1}
Assume that there exists $c>0$ such that
\begin{equation}\label{A-1}
    F(s)\leq cG^{'}(s)G(s)
\end{equation}
Then the assertion of Proposition \ref{cacciopoly} holds.
\end{lemma}
\begin{proof}
By a direct computation,
\begin{align}
\nabla u\cdot \nabla  \left( \theta^{2}F(u)\right)
& = \theta^{2}F^{'}(u) \vert\nabla u \vert^{2} + 2F(u)\nabla u\cdot \theta \nabla \theta \nonumber\\
& = \theta^{2}F^{'}(u)\vert\nabla u \vert ^{2} + 2\dfrac{F(u)}{G^{'}(u)}
G^{'}(u)\nabla u\cdot \theta\nabla \theta.
\end{align}
The right-hand side of above we transform as follows.
\begin{align}
&\theta^{2}F^{'}(u) \vert\nabla u \vert^{2} + 2F(u)\nabla u\cdot \theta \nabla \theta\\
&= \theta^{2}\vert G^{'}(u)\vert^{2} \vert\nabla u\vert^2 + 2\frac{F(u)}{G^{'}(u)}  \left( G^{'}(u) \theta\nabla u   + G(u) \nabla \theta \right) \cdot \nabla \theta -2\frac{F(u)}{G^{'}(u)}G(u)|\nabla \theta|^{2} \nonumber \\
& = |\nabla(\theta G(u))-G(u)\nabla \theta|^2
+ 2 \dfrac{F(u)}{G^{'}(u)} \nabla(\theta G(u)) \cdot \nabla \theta -2 \dfrac{F(u)}{G^{'}(u)} G(u)\vert\nabla\theta\vert^{2} \nonumber\\
&=|\nabla(\theta G(u))|^{2}
+2\left[  \dfrac{F(u)}{G^{'}(u)}- G(u) \right] \nabla \theta \cdot\nabla(\theta G(u))\nonumber - 2\left[ \dfrac{F(u)}{G^{'}(u)} - G(u) \right]|\nabla \theta |^{2}G(u). \nonumber
\\
&\geq |\nabla(\theta G(u))|^{2}
- \left \vert 2\left[  \dfrac{F(u)}{G^{'}(u)}- G(u) \right]\nabla \theta \cdot\nabla(\theta G(u)) \right \vert
-2\left[ \dfrac{F(u)}{G^{'}(u)} - G(u) \right]|\nabla \theta |^{2}G(u).
\label{before-cauchy}\end{align}
By the Cauchy Inequality,
\begin{align}
2\left[  \dfrac{F(u)}{G^{'}(u)}- G(u) \right]\nabla \theta \cdot\nabla(\theta G(u))
\leq
2\left[  \dfrac{F(u)}{G^{'}(u)}- G(u) \right]^{2} |\nabla \theta |^{2} + \dfrac{1}{2}|\nabla(\theta G(u))|^{2}. \nonumber
\end{align}
Then \eqref{before-cauchy} becomes
\begin{align*}
\nabla u\cdot \nabla\left[  \theta^{2}F(u)\right] & \geq \dfrac{1}{2}|\nabla(\theta G(u))|^{2}
- 2\left[  \dfrac{F(u)}{G^{'}(u)}- G(u) \right]^{2}\cdot |\nabla \theta |^{2}
-\left[ 2 \dfrac{F(u)}{G^{'}(u)} - G(u) \right]|\nabla \theta |^{2}G(u) \nonumber \\
& = \dfrac{1}{2}|\nabla(\theta G(u))|^{2}- \left[  2 \left[\dfrac{F(u)}{G^{'}(u)}\right]^{2} -2\dfrac{F(u)}{G^{'}(u)} G(u) + G^{2}(u) \right] |\nabla \theta|^{2}
\end{align*}
Now observe that, with $L(u)=\frac{F(u)}{G(u)G'(u)}$,
\[
\left[  2 \left[\dfrac{F(u)}{G^{'}(u)}\right]^{2} -2\dfrac{F(u)}{G^{'}(u)} G(u) + G^{2}(u) \right]
= \left[2 L^2(u) - 2L(u) +1  \right]G^{2}(u)
\]
Finally, since $0\leq L(u)\leq c$ by \eqref{A-1}, we have, for some $C>1$,
\[
\frac34 \leq 2 L^2(u) - 2L(u) +1 \leq C \triangleq \max\{1, 1+2c^2-2c\}.
\]
\end{proof}
To show that \eqref{adecr} implies \eqref{A-1}, we will recall the definition of
equivalent functions.
\begin{definition}
Functions $f$ and $g$ on a set $E$ are {\it equivalent} ($f\asymp  g$) on $E$, if there exists a constant $k\geq1$ such that $ \displaystyle {k}^{-1}  g(x) \leq f(x) \leq k \ g(x)$ for all $x\in E$.
\end{definition}
\begin{proposition} \label{P-3}
Assumption \eqref{adecr} implies \eqref{A-1}.
\end{proposition}
\begin{proof} By direct computation, it follows from \eqref{F-result} that
\begin{align}
  F(s) \asymp  \  &      \ s^{-1}I^{-\frac{1}{\Lambda}-1}(s) \label{R-1} \\
  F^{'}(s) \asymp  \  &  \  \frac{F(s)}{s}[a(s)I(s)]^{-1} \label{R-2}.
\end{align}
Using \eqref{R-1}, \eqref{R-2} and  \ref{G-def}
\begin{align}
\frac{F(s)}{ [G(s)G^{'}(s)]} & =
    \frac{\displaystyle F(s)}{\displaystyle \left(\int_{0}^{s} \sqrt{F^{'}(t) \ dt} \right)\displaystyle  \left( \sqrt{F^{'}(s)}\right) }\\
    & \asymp \frac{\displaystyle F(s)}{\displaystyle \left( \int_{0}^{s} \sqrt{\displaystyle \frac{F(t)}{t}[a(t)I(t)]^{-1}} \ dt\right)\left( \sqrt{\displaystyle \frac{F(s)}{s}[a(s)I(s)]^{-1}}\right) }\\
     & =
    \frac{\displaystyle [I(s)]^{-\frac{1}{2\Lambda}-\frac{1}{2}} [a(s)I(s)]^{\frac{1}{2}}}{\displaystyle \int_{0}^{s} t^{-1}[I(t)]^{-\frac{1}{2\Lambda}-\frac{1}{2}} [a(t)I(t)]^{-\frac{1}{2}} \ dt}\\
     & =
    \frac{\displaystyle [I(s)]^{-\frac{1}{2\Lambda} - \frac\mu2} [a(s)I^{\mu}(s)]^{\frac{1}{2}}}{\displaystyle \int_{0}^{s} t^{-1}[I(t)]^{-\frac{1}{2\Lambda}-1} [a(t)]^{-\frac{1}{2}} \ dt} \label{after 87}\ .
\end{align}
By the Cauchy's mean value theorem,
there exists $t\in(0,s)$ such that
\[
\frac{\displaystyle [I(s)]^{-\frac{1}{2\Lambda} - \frac\mu2}}{\displaystyle \int_{0}^{s} t^{-1}[I(t)]^{-\frac{1}{2\Lambda}-1} [a(t)]^{-\frac{1}{2}} \ dt}
= \left(\frac{1}{2\Lambda} + \frac\mu2\right)
\frac{\displaystyle [I(t)]^{-\frac{1}{2\Lambda} - 1 - \frac\mu2}[ta(t)]^{-1}}{\displaystyle t^{-1}[I(t)]^{-\frac{1}{2\Lambda}-1} [a(t)]^{-\frac{1}{2}}}
\asymp [a(t)I^{\mu}(t)]^{-\frac12}\ .
\]
Since \eqref{adecr} holds, one has
\begin{equation}
 \frac{F(s)}{G(s)G^{'}(s)} \asymp
 \left[\frac{a(s)I^{\mu}(s)}{a(t)I^{\mu}(t)}\right]^{\frac12}
\leq \frac1{\sqrt{c}} \ ; \  0<t<s .
\end{equation}
\end{proof}
A sufficient condition for \eqref{adecr} is given in the next remark.
\begin{remark}\label{I-B-prop}
Assume that there exists $\tilde a\in C^1(0,\infty)$ such that
 $a\asymp \tilde a$ on $(0,\infty)$, and
 \begin{equation}\label{R-4}
     \limsup_{s \to 0}   s\tilde I(s) \tilde a^{'}(s) <\infty,
 \end{equation}
\text{ where } $ \displaystyle \tilde I(s) \triangleq  \int_s^\infty \frac{dt}{t\tilde a(t)}$. Since we a free in choosing a behaviour of $P$ and $\tilde a$ at infinity, we may assume
$\limsup_{s \to \infty}   s\tilde I(s) \tilde a^{'}(s) <\infty$ as well, so function
 $s\mapsto s\tilde I(s) \tilde a^{'}(s)$ is bounded above,
 \[
 B=\sup\limits_{s>0}s\tilde I(s) \tilde a^{'}(s).
 \]
 Fix  $\mu\geq B$. Then
 the function $Q(s) = \tilde a(s)\tilde I^{\mu}(s)$ is non-increasing since
 \[
 Q'(s) = \tilde a^{'}(s)\tilde I^{\mu}(s) - \mu s^{-1}\tilde I^{\mu-1}(s) = s^{-1}\tilde I^{\mu-1}(s)
 \left[s\tilde I(s) \tilde a^{'}(s) - \mu\right]\leq 0.
 \]
 Finally, $a I^{\mu}\asymp\tilde a\tilde I^{\mu}$.
\end{remark}



\section{A Proof of Finite Propagation Speed}\label{localization}
\begin{theorem}\label{Main-T}
Let $a$ be as in Hypothesis \ref{tau}, and let \eqref{a} and \eqref{adecr} hold. Choose 
$\Lambda$ as in \eqref{Lambda} and $H,F,G$ as in \eqref{H-choice}, \eqref{F-result} and
\eqref{G-def}, respectively.

For a domain $\Omega\in\mathbb{R}^N$ and $T>0$,
let $u$ be a positive bounded weak solution to \eqref{M-2}  in $\Omega_T$, i.e. $u$ be as in Definition \ref{solution} satisfying \eqref{weak}.

Then $u$ enjoys a finite propagation speed property in the sense of \eqref{FPS}. 
\end{theorem}

\begin{proof}
Let $B\Subset\Omega$ be an open ball and let $\epsilon\in(0,1)$. Assume that $u(x,0)=0$ for all $x\in B$. We are to construct $T'\in(0,T)$ such that
$u=0$ on $\epsilon B\times[0,T'].$

We start from proving the following estimate. With $C$ as is \eqref{lemma-1-R}, for every 
$\theta\in Lip_c(B)$, one has
\begin{align}
\int_{B}  \theta^{2}H(u)\ dx  +
\frac{1}{2} \int_{0}^{t}\int_{B}| \nabla(\theta G(u))|^{2}\ dxd\tau
 \leq
C\int_{0}^{t}\int_B G^{2}(u) |\nabla \theta|^2\ dxd\tau. \label{fsp-2}
\end{align}
Indeed, by Proposition \ref{cacciopoly}, due to assumption \eqref{adecr}, we can apply
\eqref{lemma-1-R} to the left hand side of \eqref{weak} with $\phi(x,\tau)=\theta^2(x)\zeta(\tau)$, where $\zeta\in Lip_c(0,t),$ $0\leq\zeta\leq1$, approximating $\ind_{(0,t)}$ in $BV[0,T]$.
Then
\begin{align*}
\frac{1}{2} \int_{0}^{t}\int_{B}| \nabla(\theta G(u))|^{2}\zeta\ dxd\tau   
\leq C\int_{0}^{t}\int_B G^{2}(u) |\nabla \theta|^2\zeta\ dxd\tau 
+ \int_0^t \int_B \theta^2H(u)\ dx\ \zeta' d\tau.
\end{align*}
Note that, by Definition \ref{solution}, map $\tau\mapsto \int_B \theta^2H\big(u(\tau)\big)\ dx$ is continuous, and that
$\zeta' d\tau \to \delta_0 - \delta_t$. Hence
\[
\int_0^t \int_B \theta^2H(u)\ dx\ \zeta' d\tau
\to \int_B \theta^2H\big(u(0)\big)\ dx
- \int_B \theta^2H\big(u(t)\big)\ dx
= - \int_B \theta^2H\big(u(\tau)\big)\ dx.
\]
(The last equality is yielded by $u(x,0)=0$ for $x\in B$.) Hence \eqref{fsp-2}. 

In particular, \eqref{fsp-2} implies that  $\nabla G(u) \in L^2_{loc}(\Omega\times [0,T])$. 
So, by Proposition \ref{parabolic_Sobolev}, due to assumption \eqref{a}, we can apply \eqref{lemma-2-R} to the right hand side of \eqref{fsp-2} and get the following. 
For $\tilde\theta, \hat\theta \in Lip_c(B)$,
$\hat\theta =1$ on $\supp \tilde\theta$, one has
\begin{align*}
\int_B  \tilde\theta^{2}H\big(u(t)\big)\ dx & +
\frac{1}{2} \int_{0}^{t}\int_{\Omega}| \nabla(\tilde\theta G(u))|^{2} \ dxd\tau\\
   & \leq D\|\nabla\tilde\theta\|_\infty^2
    t^{1-k(1+j)}
    \left[\sup_{0\leq \tau \leq t} \int_B\hat\theta^{2}H\big(u(\tau)\big) dx + \int_{0}^{t} \int_B |\nabla(\hat\theta G(u))|^{2} dx  d\tau\right]^{1+jk},
\end{align*}
with $D=C S^{k(1+j)}$. Take supremum in $t\in(0,s)$
for $0<s\leq T$, to obtain
\begin{equation*}
\begin{aligned}
 \sup_{0\leq \tau \leq s}& \int_B\tilde\theta^{2}H\big(u(\tau)\big) dx + \int_{0}^{s} \int_B |\nabla(\tilde\theta G(u))|^{2} dx  d\tau  
 \\
 \leq & D\|\nabla\tilde\theta\|_\infty^2
    s^{1-k(1+j)}
    \left[\sup_{0\leq \tau \leq s} \int_B\hat\theta^{2}H\big(u(\tau)\big) dx + \int_{0}^{s} \int_B |\nabla(\hat\theta G(u))|^{2} dx  d\tau\right]^{1+jk}.
 \end{aligned}
\end{equation*}
Multiply the latter with $s^\beta$ with $\beta > 0$ such that
\[
\beta +1-(1+j) k = \beta(1+k j)\quad
\left(\beta \triangleq \frac{1-(1+j)k}{k j} \right).
\]
Then we get 
\begin{equation}\label{fsp-3}
\begin{aligned}
 s^\beta\sup_{0\leq \tau \leq s}& \int_B\tilde\theta^{2}H\big(u(\tau)\big) dx + s^\beta\int_{0}^{s} \int_B |\nabla(\tilde\theta G(u))|^{2} dx  d\tau  
 \\
 \leq & D\|\nabla\tilde\theta\|_\infty^2
    \left[s^\beta\sup_{0\leq \tau \leq s} \int_B\hat\theta^{2}H\big(u(\tau)\big) dx + s^\beta\int_{0}^{s} \int_B |\nabla(\hat\theta G(u))|^{2} dx  d\tau\right]^{1+jk}.
 \end{aligned}
\end{equation}
Estimate \eqref{fsp-3} is vehicle of the following iteration procedure. 

Choose $b>2$ such that 
\[\frac{b-2}{b-1}=\epsilon \quad\left(b=1+\frac1{1-\epsilon} \right)\]. 
Define
\[\epsilon_n=\frac{b-2 + b^{-n}}{b-1},\qquad n=0,1,2,\ldots\]
Obviously, 
\[\epsilon_0=1,\quad \epsilon_\infty=\epsilon,\quad \epsilon_n-\epsilon_{n+1}=b^{-(n+1)}.\]
So, for $n=0,1,2,\ldots,$ we can choose 
$\theta_n\in Lip_c(\epsilon_nB)$ such that
$\theta_n=1$ on $\epsilon_{n+1}B$ and 
\[
\|\nabla\theta_n\|_\infty^2\leq K b^{2(n+1)}
\]
with the same constant $K$ for all $n=0,1,2,\ldots$

Define
\begin{align}
Y_{n} [s] \triangleq s^{\beta} \sup_{0 \leq \tau \leq s}\int_B\theta_{n}^{2}H(u) dx +s^{\beta} \int_{0}^s \int_B |\nabla(\theta_{n}G(u))|^{2} dx d\tau. \label{En-fun}
\end{align}
Then \eqref{fsp-3} yields the iterative inequality
\begin{align}
    Y_{n+1}[s] \leq  DKb^4 \cdot ({b^2})^{n}\, Y_{n}^{1+k j}[s].
\end{align}
Finally, choose $s>0$ such that
\begin{equation}\label{X_0-assump}
Y_{0}[s]\leq (DKb^4 )^{-\dfrac{1}{k j}} b^{-\dfrac{2}{k^2 j^2}}.
\end{equation}
Then by  \cite{Lady-lemma},
\[
s^{\beta} \sup_{0 \leq \tau \leq s}\int_{\epsilon B}H(u) dx \leq \lim\limits_{n\to\infty} Y_{n}[s] = 0.
\]
\end{proof}

\begin{remark}
The use of the Ladyzhenskaya-Uraltceva iterative Lemma for the proof of finite speed of propagation for degenerate parabolic  equation was first used in \cite{ves-ted}.
\end{remark}


\section{Models for Degeneracy} \label{Exmp-P-F}
Without loss of generality, in this section, we will assume that $u \in (0,1]$, and  illustrate some generic examples of the function $a$, for which hold all constrains on the functions $F$, $G$ and $H$, with the following summarized remark.
\begin{remark}\label{P-test}
Let  $a$ and $I\in C^1(0,\infty)$ be as in Hypothesis \ref{tau}. Assume that
\begin{align}
    \limsup\limits_{s\to0} a(s)I(s) & < \infty,\\
    \limsup\limits_{s\to0} sa{'}(s)I(s) & < \infty.
\end{align}
Then Remark \ref{I-B-prop} yields Theorem \ref{Main-T} on finite speed of propagation.
\end{remark}
\begin{exmpl}\label{Ex-1}{$a(s)=s^{\beta}$,  where  $s\in [0,\infty)$ and $\beta > 0$}. \normalfont
\begin{align}
 I(s) & =  \int_{s}^\infty t ^{-\beta-1} d t =  \ {\beta^{-1}} s^{-\beta}.\\
 a(s)I(s) & =  {\beta^{-1}}.  \\
 sa'(s)I(s) & = 1.
\end{align}
\end{exmpl}
\begin{exmpl}{$ a(s)= \exp \left(-\dfrac{1}{s^{\beta}}\right), s \in [0,1), \beta > 0$}.\normalfont
\begin{align}
 I(s) = \  &  \int_{s}^{1} t^{-1}\exp \left(\dfrac{1}{t^{\beta}}\right) dt .\\
a(s)I(s)
=  \ &  {\exp\left(-\dfrac{1}{s^{\beta}}\right)}\left[\int_{s}^{1} t ^{-1}\exp  \left(\dfrac{1}{t^{\beta}}\right) dt \right].
\end{align}
By L'Hôpital's rule
\begin{align} \displaystyle
\lim_{s \to 0} a(s)I(s) \equiv  \ \lim_{s \to 0} s^{\beta} \equiv  \  0 \ .
\end{align}  
Then
\begin{align}
    sI(s)a{'}(s) =  \ &  \  s \left[\int_{s}^{1} t ^{-1}\exp \left(\dfrac{1}{t^{\beta}}\right) dt \right] \exp \left(-\frac{1}{s^{\beta}}\right)s^{\beta-1},\\
 \lim_{s \to 0} sI(s)a{'}(s)   =  & \lim_{s \to 0}
    \frac{ \displaystyle \left[\int_{s}^{1}t^{-1} \exp \left( \frac{1}{t^{\beta}}\right) dt \right]s^{-\beta}}{ \beta \displaystyle \exp \left( \frac{1}{s^{\beta}}\right)}.
\end{align}
By L'Hôpital's rule
\begin{align}
 \lim_{s \to 0} sI(s)a{'}(s)
    \equiv & \ 1 +  \lim_{s \to 0}  \frac{ \left[\displaystyle \int_{s}^{1}t^{-1} \exp \left( \frac{1}{t^{\beta}}\right) dt\right]}{ \displaystyle \exp \left( \frac{1}{s^{\beta}}\right)}  \\
  \equiv & \ 1 +  \lim_{s \to 0}  s^{\beta} \\
   = & \ 1.
\end{align}
\end{exmpl}
\vspace{-0.4 cm}
\begin{exmpl}{$ \displaystyle a(s)= \exp \left(-\int_{s}^{1} \frac{\zeta(\tau)}{\tau} d\tau\right), s \in (0,1]$} \normalfont and $ 0 < k_{1} < \zeta(s) < k_{2}$.  \begin{align}
    \displaystyle I(s) =\int_{s}^{1} t^{-1}\exp \left(\int_{t}^{1} \frac{\zeta(\tau)}{\tau} d\tau\right) \ dt.
    \end{align}
Then
\begin{align}
 \displaystyle a(s)I(s)
     = & \  \frac{ \displaystyle\left[\int_{s}^{1} t^{-1}\exp \left(\int_{t}^{1} \frac{\zeta(\tau)}{\tau} d\tau\right) \ dt  \right]}
     {\displaystyle\left[ \exp \left(\int_{s}^{1} \frac{\zeta(\tau)}{\tau} d\tau\right)\right]}.
\end{align}
By L'Hôpital's rule
\begin{align}
   \limsup_{s \to 0} a(s)I(s) \equiv & \  \limsup_{s \to 0} \frac{1}{\zeta(s)}  <\frac1{k_1}.  \label{ex-3}
      \end{align}
Then
\begin{align}
s I(s)a'(s) = & \ {\displaystyle\left[\int_{s}^{1} t^{-1}\exp \left(\int_{t}^{1} \frac{\zeta(\tau)}{\tau} d\tau\right) \ dt \right] }{\displaystyle \exp \left(-\int_{t}^{1} \frac{\zeta(\tau)}{\tau} d\tau\right)\zeta(s)},\\
\lim_{s \to 0} s I(s)P^{'}(s) \equiv  \ & \lim_{s \to 0}  \frac{\displaystyle\left[\int_{s}^{1} t^{-1}\exp \left(\int_{s}^{1} \frac{\zeta(\tau)}{\tau} d\tau\right) \ dt \right]}{\displaystyle\exp \left(\int_{t}^{1} \frac{\zeta(\tau)}{\tau} d\tau\right)}.
\end{align}
We apply L'Hôpital's rule to get
\begin{align}
\limsup_{s \to 0} s I(s)P'(s) =  \limsup_{s \to 0} \frac{1}{\zeta(s)}  < \frac1{k_1}.  \end{align}
\end{exmpl}
\begin{exmpl}\label{Ex-4}{\normalfont Let $\zeta$  be such that $\displaystyle 0 < \ k_{3} \leq \frac{\zeta}{\zeta_{0}} \leq k_{4}$  for some $\zeta_{0}^{'} \leq 0$  and $\displaystyle \sup_{0<s<1}{s|\zeta_{0}^{'}|}{\zeta_{0}^{-1}} =c_{0} < \infty$. (Note that  $ \displaystyle \lim_{\tau \to 0} \zeta(\tau) = \infty$.) Let $ \displaystyle a(s)= \exp \left(-\int_{s}^{1} \frac{\zeta(\tau)}{\tau} d\tau\right), \ s \in (0,1] $}.\normalfont  \vspace{0.2 cm}\\
Similarly to \eqref{ex-3} we get
\begin{align}
\displaystyle  \lim_{s \to 0} a(s)I(s) \equiv \lim_{s\mapsto 0} \frac{1}{\zeta(s)}  = 0.
\end{align}
Then
\begin{align}
sI(s)a'(s) =  & s \left[\int_{s}^{1} t^{-1}\exp \left(\int_{t}^{1} \frac{\zeta(\tau)}{\tau} d\tau\right) \ dt\right] \exp \left(-\int_{t}^{1} \frac{\zeta(\tau)}{\tau} d\tau\right) \zeta(s).
\end{align}
By L'Hôpital's rule
\begin{align}
\lim_{s \to 0} sI(s)a{'}(s)
\equiv   & \ 1 + \lim_{s \to 0} \frac{\left[\displaystyle \int_{s}^{1} t^{-1}\exp \left(\int_{t}^{1} \frac{\zeta(\tau)}{\tau} d\tau\right) \ dt \right] \frac{ \displaystyle s \vert \zeta_{0}^{'}(s)\vert}{ \displaystyle \zeta(s)}}{\displaystyle \exp \left(\int_{s}^{1} \frac{\zeta(\tau)}{\tau} d\tau\right) \ dt} \nonumber \\
\equiv   & \ 1 +  \lim_{s \to 0}  \frac{C_{0}}{\zeta(s)}
=  \ 1.
\end{align}.
\end{exmpl}



\end{document}